\documentclass[11pt,leqno]{amsart}

\newtheorem{theorem}{\sc Theorem}[section]
\newtheorem{lemma}[theorem]{\sc Lemma}

\newtheorem{corollary}[theorem]{\sc Corollary}

\title[Coprime commutators]{ On finite groups in which coprime commutators are covered by few cyclic subgroups}
\author{Cristina Acciarri}

\address{Cristina Acciarri:  Department of Mathematics, University of Brasilia,
Brasilia-DF, 70910-900 Brazil}
\email{acciarricristina@yahoo.it}

\author{Pavel Shumyatsky} 

\address{Pavel Shumyatsky: Department of Mathematics, University of Brasilia,
Brasilia-DF, 70910-900 Brazil}

\email{pavel@unb.br}


\thanks{This research was supported by CNPq}

\begin{document}
\begin{abstract} 
The coprime commutators $\gamma_j^*$ and $\delta_j^*$ were recently introduced  as a tool to study properties of finite groups that can be expressed in terms of commutators of elements of coprime orders. They are defined as follows.  Let $G$ be a finite group. Every element of $G$ is  both a $\gamma_1^*$-commutator and a $\delta_0^*$-commutator. Now let $j\geq 2$ and let $X$ be the set of all elements of $G$ that are powers of $\gamma_{j-1}^*$-commutators. An element $g$ is a $\gamma_j^*$-commutator if there exist $a\in X$ and $b\in G$ such that $g=[a,b]$ and $(|a|,|b|)=1$. For $j\geq 1$ let $Y$ be the set of all elements of $G$ that are powers of $\delta_{j-1}^*$-commutators. The element $g$ is a $\delta_j^*$-commutator if there exist $a,b\in Y$ such that $g=[a,b]$ and $(|a|,|b|)=1$. The subgroups of $G$ generated by all $\gamma_j^*$-commutators and all  $\delta_j^*$-commutators are denoted by $\gamma_j^*(G)$ and $\delta_j^*(G)$, respectively. For every $j\geq2$ the subgroup $\gamma_j^*(G)$ is precisely the last term of the lower central series of $G$ (which throughout the paper is denoted by $\gamma_\infty(G)$) while for every $j\geq1$ the subgroup $\delta_j^*(G)$ is precisely the last term of the lower central series of $\delta_{j-1}^*(G)$, that is, $\delta_j^*(G)=\gamma_\infty(\delta_{j-1}^*(G))$. 

In the present paper we prove that if $G$ possesses $m$ cyclic subgroups whose union contains all $\gamma_j^*$-commutators of $G$, then $\gamma_j^*(G)$ contains a subgroup $\Delta$, of $m$-bounded order, which is normal in $G$ and has the property that $\gamma_{j}^{*}(G)/\Delta$ is cyclic. If $j\geq2$ and $G$ possesses $m$ cyclic subgroups whose union contains all $\delta_j^*$-commutators of $G$, then the order of $\delta_j^*(G)$ is $m$-bounded.

\end{abstract}

\maketitle

%%%%%%%%%%%%%%%%%%%%%%%%%%%%%%%%%%%%%%%%%%%%%%%%
%%%%%%%%%%%%%%%%%%%%%%%%%%%%%%%%%%%%%%%%%%%%%%%%
\section{Introduction}
A covering of a group $G$ is a family $\{S_i\}_{i\in I}$ of subsets of
$G$ such that $G=\bigcup_{i\in I}\,S_i$.
If $\{H_i\}_{i\in I}$ is a covering of $G$ by subgroups, it is natural to ask
what information about $G$ can be deduced from properties of the subgroups $H_i$. In the case where the covering is finite actually quite a lot about the structure of $G$ can be said. In particular, as was first pointed out by Baer (see \cite[p.\ 105]{Robinson}), a group covered by finitely many cyclic subgroups is either cyclic or finite. More recently Fern\'andez-Alcober and Shumyatsky proved that if $G$ is a group in which the set of all commutators is covered by finitely many cyclic subgroups, then $G'$ is either finite or cyclic \cite{FS}. This suggests the question about the structure of a group in which the set of all $\gamma_j$-commutators (or of all $\delta_j$-commutators) is covered by finitely many cyclic subgroups. Here the words $\gamma_j$ and $\delta_j$ are defined by the positions $\gamma_1=\delta_0=x_1$, $\gamma_{j+1}=[\gamma_j,x_{j+1}]$ and $\delta_{j+1}=[\delta_j,\delta_j]$.

In \cite{CN} Cutolo and Nicotera showed that if $G$ is a group in which the set of all $\gamma_{j}$-commutators is covered by finitely many cyclic subgroups, then $\gamma_{j}(G)$ is finite-by-cyclic. They also showed that $\gamma_{j}(G)$ can be neither cyclic nor finite. It is still unknown whether  a similar result holds for the derived words $\delta_{j}$.

In \cite{forum} the coprime commutators $\gamma_j^*$ and $\delta_j^*$ were introduced  as a tool to study properties of finite groups that can be expressed in terms of commutators of elements of coprime orders.  For the  reader's convenience we recall here the definitions. Let $G$ be a finite group. Every element of $G$ is  both a $\gamma_1^*$-commutator and a $\delta_0^*$-commutator. Now let $j\geq 2$ and let $X$ be the set of all elements of $G$ that are powers of $\gamma_{j-1}^*$-commutators. An element $g$ is a $\gamma_j^*$-commutator if there exist $a\in X$ and $b\in G$ such that $g=[a,b]$ and $(|a|,|b|)=1$. For $j\geq 1$ let $Y$ be the set of all elements of $G$ that are powers of $\delta_{j-1}^*$-commutators. The element $g$ is a $\delta_j^*$-commutator if there exist $a,b\in Y$ such that $g=[a,b]$ and $(|a|,|b|)=1$. The subgroups of $G$ generated by all $\gamma_j^*$-commutators and all  $\delta_j^*$-commutators will be denoted by $\gamma_j^*(G)$ and $\delta_j^*(G)$, respectively. One can easily see that if $N$ is a normal subgroup of $G$ and $x$ an element whose image in $G/N$ is a $\gamma_j^*$-commutator (respectively a $\delta_j^*$-commutator), then there exists a $\gamma_j^*$-commutator $y$ in $G$ (respectively a $\delta_j^*$-commutator) such that $x\in yN$. 

It was shown in \cite{forum} that $\gamma_j^*(G)=1$ if and only if $G$ is nilpotent and $\delta_j^*(G)=1$ if and only if the Fitting height of $G$ is at most $j$. It follows that for every $j\geq2$ the subgroup $\gamma_j^*(G)$ is precisely the last term of the lower central series of $G$ (which throughout the paper will be denoted by $\gamma_\infty(G)$) while for every $j\geq1$ the subgroup $\delta_j^*(G)$ is precisely the last term of the lower central series of $\delta_{j-1}^*(G)$, that is, $\delta_j^*(G)=\gamma_\infty(\delta_{j-1}^*(G))$ . 

In the present paper we prove the following theorem.
\begin{theorem}\label{gammathm} Let  $j$ be a positive integer and $G$  a finite group that possesses $m$ cyclic subgroups whose union contains all $\gamma_j^*$-commutators of $G$. Then $\gamma_j^*(G)$ contains a subgroup $\Delta$, of $m$-bounded order, which is normal in $G$ and has the property that $\gamma_{j}^{*}(G)/\Delta$ is cyclic.
\end{theorem}

We note that the above result seems to be new even in the case where $j=1$. Thus, one immediate corollary of Theorem \ref{gammathm} is that a finite group covered by $m$ cyclic subgroups has a normal subgroup $\Delta$ of $m$-bounded order with the property that $G/\Delta$ is cyclic. This can be easily extended to arbitrary groups.

\begin{corollary} Let $G$ be a (possibly infinite) group covered by $m$ cyclic subgroups. Then $G$ has a finite normal subgroup $\Delta$, of $m$-bounded order, such that $G/\Delta$ is cyclic.
\end{corollary}
Indeed, let $G$ be as in the above corollary. The classical result of B.\,H. Neumann \cite{neumann} tells us that $G$ has a cyclic subgroup of finite index. Therefore $G$ is residually finite and all finite quotients of $G$ satisfy the hypothesis of Theorem \ref{gammathm}. Hence, $G$ has a normal subgroup $\Delta$ of $m$-bounded order with the property that $G/\Delta$ is cyclic. 

We also mention that in Theorem \ref{gammathm} the subgroup $\gamma_j^*(G)$ is (of bounded order)-by-cyclic and so we observe here a phenomenon related to what was proved by  Cutolo and Nicotera for the verbal subgroups $\gamma_{j}(G)$.  
 
Having dealt with Theorem \ref{gammathm}, it is natural to look at finite groups in which $\delta_{j}^{*}$-commutators can be covered by few cyclic subgroups. Since for $j\leq1$ any $\delta_{j}^{*}$-commutator is a $\gamma_{j+1}^{*}$-commutator, the interesting cases occur when $j\geq2$.  

\begin{theorem}\label{deltathm} 
 Let  $j\geq2$ and $G$ be a finite group that possesses $m$ cyclic subgroups whose union contains all $\delta_j^*$-commutators of $G$.  Then the order of $\delta_j^*(G)$ is $m$-bounded.
\end{theorem}

Throughout the paper we use the expression ``$(a,b,\dots)$-bounded" to mean that the bound is a function of the parameters $a,b,\dots$. Henceforth all groups considered in this paper will be finite and the term ``group" will mean ``finite group".

%%%%%%%%%%%%%%%%%%%%%%%%%%%%%%%%%%%%%%%%%%%%%%%%%%%%%%%
\section{Preliminaries} 
We begin with some results about coprime actions of groups. Let $H$ and $K$ be subgroups of a group $G$. We denote by $[K,H]$ the subgroup of $G$ generated by $\{[k,h]:k\in K, h\in H\}$, and by $[K,_iH]$ the subgroup $[[K,_{i-1}H],H]$ for $i\ge2$.  If $G$  is a $p$-group, we denote by $\Omega_{1}(G)$ the subgroup of $G$ generated by its elements of order $p$.

\begin{lemma}[\cite{Gorenstein} Theorems 5.2.3, 5.2.4 and 5.3.6]
\label{Lemmapre}
Let $A$ and $G$ be groups with  $(|G|,|A|)=1$ and suppose that $A$ acts on $G$. Then we have
\begin{itemize}
\item[(1)] $[G,A,A]=[G,A]$;
\item[(2)] If $G$ is  an abelian $p$-group, then $G=C_{G}(A)\times [G,A]$;
\item[(3)] If $G$ is an abelian $p$-group and $A$ acts trivially on $\Omega_{1}(G)$, then $A$ acts trivially on $G$.
\end{itemize}
\end{lemma}

\begin{lemma}\label{cycliccom}
Let $G$ be an abelian $p$-group  and $\alpha$ a coprime automorphism of $G$. If $[G,\alpha]$ is cyclic, then $[G,\alpha]=[G,\alpha^{i}]$ for any integer $i$ such that $\alpha^{i}\ne 1$. 
\end{lemma}
\begin{proof}
By  Lemma \ref{Lemmapre}(2) we have $G=C_{G}(\alpha)\times [G,\alpha]$. Suppose that $\alpha^i\neq1$ and $[G,\alpha]\neq[G,\alpha^{i}]$. Then $C_{[G,\alpha]}(\alpha^i)\neq1$. Since $[G,\alpha]$ is cyclic, we conclude that $\Omega_1([G,\alpha])\leq C_{[G,\alpha]}(\alpha^i)$ and therefore $\alpha^i$ acts trivially on $[G,\alpha]$. This implies that $\alpha^i=1$, a contradiction. 
\end{proof}

\begin{lemma}\label{autocyclic} Let $G$ be a cyclic group faithfully acted on by a group $A$. The following holds.
\begin{itemize}
\item[(1)] The group $A$ is abelian;
\item[(2)] If $G$ is a $p$-group and $A$ is a $p'$-group, then $A$ is cyclic.
\end{itemize}
  \end{lemma}

\begin{proof} Both claims are immediate from the well-known fact that the group of automorphisms of the additive cyclic group $\Bbb Z/n\Bbb Z$ is isomorphic with the multiplicative group $(\Bbb Z/n\Bbb Z)^*$.
\end{proof}

\begin{lemma}\label{coradical}
Let $j\geq 2$ and  $G$ be a group containing a normal subgroup $N$. If $N\leq\delta_{j}^{*}(G)$  and  $\delta_{j}^{*}(G)/N$ is cyclic, then $\delta_{j}^{*}(G)=N$.  
\end{lemma}
\begin{proof} We pass to the quotient $G/N$ and without loss of generality assume that $N=1$. Therefore $\delta_{j}^{*}(G)$ is cyclic and so by Lemma \ref{autocyclic}(1) we have $\delta_{j}^{*}(G)\leq Z(G')$. It follows that $\delta_{j-1}^{*}(G)$ is nilpotent and, since $\delta_{j}^{*}(G)=\gamma_\infty(\delta_{j-1}^{*}(G))$, we deduce that $\delta_{j}^{*}(G)=1$. This completes the proof.
\end{proof}

The following lemma is well-known. The proof can be found for example in \cite{AST}.

\begin{lemma}\label{metaHall}
Let $G$ be a metanilpotent group, $P$  a Sylow $p$-subgroup of $\gamma_{\infty}(G)$ and $H$ a Hall $p'$-subgroup of $G$. Then $P=[P,H]$.
\end{lemma}
 
The next lemma will be very useful.
\begin{lemma}\label{1113} Let $y_1,\dots,y_{j+1}$ be powers of $\delta_j^*$-commutators in $G$. Suppose that the elements $y_1,\dots, y_{j+1}$ normalize a subgroup $N$ such that $(|y_i|,|N|)=1$ for every $i=1,\dots,j+1$. Then for every $g\in N$ the element $[g,y_1,\dots,y_{j+1}]$ is a $\delta_{j+1}^*$-commutator.
\end{lemma}
\begin{proof} We note that all elements of the form $[g,{y_1,\dots,y_i}]$ are of order prime to $|y_{i+1}|$. An easy induction on $i$ shows that whenever $i\leq j$ the element $[g,y_1,\dots,y_{i+1}]$ is a $\delta_{i+1}^*$-commutator. The lemma follows.
\end{proof}

 \begin{lemma}\label{copgen}
 
 Let $G$  be a group, $P$ a normal $p$-subgroup of $G$ and $x$ a $p'$-element in $G$. Let $j\geq 1$ be an integer. Then we have

\begin{itemize}
\item [(1)] The subgroup $[P,x]$ is generated by $\gamma_{j}^{*}$-commutators. 
\item [(2)] If $P$ is abelian, then every element of $[P,x]$ is a $\gamma_{j}^{*}$-commutator.
\item[(3)] If $x$ is a power of a $\delta_{j-1}^{*}$-commutator, then $[P,x]$ is generated by $\delta_{j}^{*}$-commutators.
\item[(4)] If $x$ is a power of a $\delta_{j-1}^{*}$-commutator and $P$ is abelian, then every element of $[P,x]$ is a $\delta_{j}^{*}$-commutator.
\end{itemize}
\end{lemma}

\begin{proof} In view of Lemma \ref{Lemmapre}(1) $[P,x]=[P,\underbrace{x,\dots,x}_{j-1}]$. Suppose first that $P$ is abelian. Note that every element of the form $[g,\underbrace{x,\dots,x}_{j-1}]$, with $g\in P$, is a $\gamma_{j}^{*}$-commutator. Since $P$ is abelian, every element of $[P,x]$ is of the form $[g,\underbrace{x,\dots,x}_{j-1}]$ for a suitable $g\in P$ and therefore every element of $[P,x]$ is a $\gamma_{j}^{*}$-commutator. Now drop the assumption that $P$ is abelian. We wish to show that $[P,x]$ is generated by $\gamma_{j}^{*}$-commutators. Passing to the quotient $G/\Phi(P)$ we may assume that $P$ is elementary abelian and use the result for the abelian case.  This proves Claims (1) and (2).

The proof of  Claims (3) and (4) follows a similar argument using Lemma \ref{1113}.
\end{proof}

The well-known Focal Subgroup Theorem \cite[Theorem 7.3.4]{Gorenstein} states that if $G$ is a group and $P$ a Sylow $p$-subgroup of $G$, then $P\cap G'$ is generated by the set of commutators $\{ [g,z] \mid g\in G,\ z\in P,\ [g,z]\in P \}$. In particular, it follows that $P\cap G'$ can be generated by commutators  lying in $P$.  This observation led to the question on generation of Sylow subgroups of verbal subgroups of finite groups.  The main result of \cite{focal} is that $P\cap w(G)$ is generated by powers of $w$-values, whenever $w$ is a multilinear commutator word. More recently an analogous result on the generation of Sylow subgroups of $\delta^*_j(G)$ in the case where $G$ is soluble was proved in \cite{AST}.  More precisely we have the following lemma that  we will need later on.

\begin{lemma}[\cite{AST}, Lemma 2.6] 
\label{foca} Let $j\ge 0$. Let $G$ be a soluble group and $P$ a Sylow $p$-subgroup of $G$. Then $P\cap \delta^*_j(G)$ is generated by powers of $\delta^*_j$-commutators.
\end{lemma}

It is natural to conjecture that Lemma \ref{foca} actually holds for all finite groups. In particular, the corresponding result in \cite{focal} was proved without the assumption that $G$ is soluble. It seems though that proving Lemma \ref{foca} for arbitrary groups is a complicated task. Indeed, one of the tools used in \cite{focal} is the proof of the Ore Conjecture by  Liebeck, O'Brien,  Shalev, and  Tiep \cite{lost} that every element of any nonabelian finite simple group is a commutator. Recently it was conjectured in \cite{forum} that every element of a finite simple group is a commutator of elements of coprime orders. If this is confirmed, proving Lemma \ref{foca} for arbitrary groups would be easy. However the conjecture that every element of a finite simple group is a commutator of elements of coprime orders is proved only for the alternating groups \cite{forum} and the groups ${\rm PSL}(2,q)$ \cite{pellegrini}.

\begin{lemma}\label{pri} Let $G$ be a  noncyclic $p$-group that can be covered by $m$ cyclic subgroups. Then $|G|$ is $m$-bounded.
\end{lemma}
\begin{proof} To start with, we consider the case where $G$ is abelian. We notice that the minimal number of generators of $G$ is at most $m$ and therefore it is sufficient to bound the exponent of $G$. The group $G$ contains an elementary abelian subgroup, say $J$, of order $p^2$. One requires precisely $p+1$ cyclic subgroups to cover $J$. Hence $p+1\leq m$. Let the exponent of $G$ be $p^n$. Since $p\leq m-1$, it is sufficient to bound $n$. We assume that $n\geq 2$. Choose an element $a\in G$ whose order is $p^n$ and an element $b\in G\setminus\langle a\rangle$ of order $p$. Set $H=\langle a,b\rangle$. It is clear that any covering of $H$ by cyclic subgroups requires some subgroups of order $p^n$. Further, the element $a^pb$ has order $p^{n-1}$ and it is not contained in any cyclic subgroup of order $p^n$. Therefore any covering of $H$ by cyclic subgroups requires  also some subgroups of order $p^{n-1}$. Assuming that $n\geq 3$ we now consider the element $a^{p^2}b$. This has order $p^{n-2}$ and is not contained in any cyclic subgroup of order $p^{n-1}$. Thus any covering of $H$ by cyclic subgroups requires some subgroups of order $p^{n-2}$. It now becomes clear that any covering of $H$ by cyclic subgroups requires some subgroups of all possible orders $p^n,p^{n-1},\dots,p$. It follows that $n\leq m$ and in the case where $G$ is abelian the lemma is proved.

We now drop the assumption that $G$ is abelian. Let $N$ be a maximal normal abelian subgroup. Then $N=C_G(N)$. If $N$ is noncyclic, then by the previous  argument  $|N|$ is $m$-bounded and, since $G/N$ embeds in $Aut\, N$,  the order of $G$ is $m$-bounded, too. Hence we assume that $N$ is cyclic of order $p^n$. The quotient $G/G'$ is abelian and noncyclic. Hence $G/G'$ contains an elementary abelian subgroup of order $p^2$. We have remarked in the previous paragraph that the existence of such a subgroup implies that $p\leq m-1$ and so now it is sufficient to bound $n$. Let $y$ be an element of least order in $G\setminus N$. In view of \cite[Theorem 5.4.4]{Gorenstein} the order of $y$ is either $p$ or 4. Let $P=N\langle y\rangle$. Since $C_P(y)$ is abelian, the previous paragraph shows that $|C_P(y)|$ is $m$-bounded. Hence, it is sufficient to bound the index of $C_N(y)$ in $N$. This is precisely the order of the subgroup $[N,y]$. Observe that all elements in the coset $[N,y]y^{-1}$ are conjugate to $y^{-1}$ and so $P$ contains at least $|[N,y]|$ elements of order $|y|$ (which is either $p$ or 4). Any nontrivial cyclic $p$-group contains exactly $p-1$ elements of order $p$ and at most two elements of order 4. Therefore one requires at least $|[N,y]|/p$ cyclic subgroups in $P$ to cover the coset $[N,y]y^{-1}$. Hence $|[N,y]|/p\leq m$ and since $p\leq m-1$, we deduce that $|[N,y]|\leq m(m-1)$. The proof is complete.
\end{proof}

We close this preliminary section with the following results about coprime actions.

\begin{lemma}
\label{coctionbounded}
Let $j$ be a positive integer, $P$ a $p$-group of class $c$ and $\alpha$ a $p'$-automorphism of $P$. Suppose that $P$ has $m$ cyclic subgroups whose union contains all elements of the form $[x,\underbrace{\alpha,\dots,\alpha}_{j}]$, with $x\in P$. If $[P,\alpha]$ is noncyclic, then the order of $[P,\alpha]$ is $(c,m)$-bounded. 
\end{lemma}

\begin{proof}
By Lemma \ref{Lemmapre}(1) we have $P=[P,\alpha]=[P,\underbrace{\alpha,\dots,\alpha}_{j}]$.  We argue by induction on the nilpotency class $c$. If $c=1$, then  $P$ is abelian and it consists of elements of the form $[x,\underbrace{\alpha,\dots,\alpha}_{j}]$.  It follows that $P$ can be covered by $m$ cyclic subgroups and by Lemma \ref{pri} the order of $P$ is $m$-bounded.

Assume $c\geq 2$ and pass to the quotient $\overline{P}=P/P'$.  Of course $\overline{P}$ is not cyclic and  abelian.  Hence  by the argument in the previous paragraph the order of $\overline{P}$ is $m$-bounded and  since $P$ is nilpotent of class $c$, it follows that  $|P|$ is $(c,m)$-bounded, as desired. 
\end{proof}

\begin{lemma}
\label{bibi} Let $A$ be a noncyclic $p'$-group of automorphisms of a noncyclic abelian $p$-group $G$. Then there exists $a\in A$ such that $[G,a]$ is noncyclic.
\end{lemma}
\begin{proof} Suppose that the lemma is false and $[G,a]$ is cyclic for every $a$ in $A$. 

Firstly we consider the case where $A$ is abelian. Choose a nontrivial element $a_1\in A$. The cyclic subgroup $[G,a_1]$ is $A$-invariant and, by Lemma \ref{autocyclic}, the quotient $A/C_A([G,a_1])$ is cyclic. In particular $C_A([G,a_1])\neq1$ so we choose a nontrivial element $a_2\in C_A([G,a_1])$. Since $a_2$ centralizes $[G,a_1]$, it follows that $[G,a_1][G,a_2]$ is not cyclic. Moreover, it is clear that $a_1$ centralizes $[G,a_2]$. Hence, $[G,a_1][G,a_2]\leq[G,a_1a_2]$ and this is a contradiction. Thus, in the case where $A$ is abelian the result follows.

Suppose now that $A$ is nilpotent. If $A$ contains  a noncyclic abelian subgroup, then the result follows from the previous paragraph. Hence, without loss of generality, we suppose that every abelian subgroup of $A$ is cyclic. It follows (see for example \cite[Theorem 4.10(ii), p.\ 199]{Gorenstein}) that $A$ is isomorphic to $Q\times C$, where $Q$ is the generalized quaternion group and $C$ is a cyclic group of odd order. By Lemma \ref{cycliccom} for any $a$ in $A$ and any integer $i$ such that $a^{i}\neq 1$ we have $[G,a]=[G,a^{i}]$. Let $a_{0}$ be the unique involution of $A$. It is clear that $a_0$ is contained in all maximal cyclic subgroups of A. It follows that $[G,a]=[G,a_{0}]$ for all $a$ in $A$. Hence we conclude that $[G,A]=[G,a_{0}]$ which is cyclic. By Lemma \ref{Lemmapre} $A$ acts faithfully on $[G,a_0]$ and, in view of Lemma \ref{autocyclic}(2), the group $A$ must be cyclic. This is a contradiction.

Finally we can drop the assumption that $A$ is nilpotent. If $A$ contains at least one  noncyclic nilpotent subgroup, we use the previous case.  Thus,  we assume that all nilpotent subgroups in $A$ are cyclic and in this case $A$ is soluble. Let $F=F(A)$ be the Fitting subgroup of $A$. Of course we can assume that $A$ is not nilpotent and so we can choose a subgroup $Q$ of $F$ of prime order $q$ such that $Q$ is not contained in $Z(A)$. Then there exists a $q'$-element $a$ in $A$ such that $[Q,a]=Q$. The element $a$ acts on $[G,Q]$, which is a cyclic $p$-group. Thus $Q\langle a\rangle$ acts on $[G,Q]$, but this leads to a contradiction since by Lemma \ref{autocyclic}(1)  the group of automorphisms of a cyclic group is abelian. 
\end{proof}

%%%%%%%%%%%%%%%%%%%%%%%%%%%%%%%%%%%%%%%%%%%%%%%%%%

%%%%%%%%%%%%%%%%%%%%%%%%%%%%%%%%%%%%%%%%%%%%%%%%%%%%%
\section{Theorem \ref{deltathm}}
Turull introduced in \cite{tur} the concept of an irreducible $B$-tower and showed that a  soluble group $G$ has Fitting height $h$ if and only if $h$ is maximal such that there exists an irreducible tower of height $h$ consisting of subgroups of $G$ (see Lemmas 1.4 and 1.9(3) in \cite{tur}). We will now remind the reader some of the properties of subgroups forming an irreducible  tower (we require only the case $B=1$ and refer to these objects simply as ``towers'').

Let $P_i$, where $i=1,\dots,h$ be subgroups of $G$ forming a tower of height $h$. Then we have

(1) $P_i$ is a $p_i$-group ($p_i$ a prime) for $i=1,\dots,h$.

(2) $P_i$ normalizes $P_j$ for $i<j$.

(3) $p_i\neq p_{i+1}$ for $i=1,\dots,h-1$.

(4) $[P_i,P_{i-1}]=P_i$ for $i=2,\dots,h$.

(5)  Let $\bar{P}_i=P_i/C_{P_i}(\bar{P}_{i+1})$ for $i=1,\dots,h-1$ and $\bar{P_{h}}=P_{h}$. Then $\phi(\phi(\bar{P}_i))=1$, $\phi(\bar{P}_i)\leq Z(\bar{P}_i)$. Moreover $P_{i-1}$ centralizes $\phi(\bar{P}_i)$ for $i=2,\dots,h$. Here $\phi$ denotes the Frattini subgroup.

In the next few lemmas we will assume that $\delta_{j+1}^{*}(G)=1$. Therefore $\delta_{j}^{*}(G)$ is nilpotent  and so any Sylow subgroups of $\delta_{j}^{*}(G)$ is normal in $G$.

\begin{lemma}\label{mumu} Let $p$ be a prime, $j$ a positive integer and $G$ a group  such that $\delta_{j+1}^*(G)=1$. Suppose that $\delta_j^*(G)$ is a nontrivial abelian $p$-group. Then  either there exists a $p'$-element $x$ which is a power of a $\delta_{j-1}^*$-commutator with the property that $[\delta_{j}^{*}(G),x]$ is noncyclic, or $\delta_{j}^{*}(G)$ is cyclic and $j=1$.
\end{lemma}
\begin{proof} For simplicity denote $\delta_{j}^{*}(G)$ by $P$. Suppose first that $P$ is cyclic. If $j\geq 2$, then in view of Lemma \ref{coradical} we deduce that $P=1$, a contradiction. Hence, if $P$ is cyclic, we have $j=1$. Now assume that $P$ is noncyclic. 

Consider the case where $j=1$. We wish to show that there exists a $p'$-element $x\in G$ with the property that $[P,x]$ is noncyclic. Let $L$ be a Hall $p'$-subgroup in $G$ and suppose that $[P,x]$ is cyclic for every $x\in L$.  If $L/C_L(P)$ is not cyclic, we obtain a contradiction with Lemma \ref{bibi}. Therefore assume that $L/C_L(P)$ is cyclic. Let $a$ be an element of $L$ such that $\langle a,C_L(P)\rangle={L}$. We have $P=[P,L]=[P,a]$, which is again a contradiction  since $[P,a]$ is cyclic.

Hence  we  may assume that $j\geq 2$. Moreover we assume that $G$ is a counter-example with minimal possible order. Since $\delta_{j+1}^*(G)=1$, it follows that  $G$ is soluble and the Fitting height precisely $j+1$. By \cite{tur} $G$ possesses a tower  of height $j+1$, i.e., a subgroup $P_0\ldots P_{j-2}P_{j-1}P_{j}$, where $P_{j}\leq P$. Again $P_j$ is noncyclic and therefore, in view of minimality of $|G|$, we have $G=P_{0}\ldots P_{j-2}P_{j-1}P_{j}$ and $P_j=P$. 

By \cite[Lemma 2.5]{forum}, each subgroup $P_{i}$ of the tower is generated by $\delta_{i-1}^{*}$-commutators contained in $P_{i}$. Set $H=P_{j-1}$. We know that $P=[P,H]$. Let $B$ be the set of all elements of $H$ which can be written as powers of $\delta_{j-1}^*$-commutators and assume that $[P,b]$ is cyclic for any $b$ in $B$. First we consider the case where $H$ has odd order.

Let $b_1,b_2$ be  elements of $B$ and $B_0=\langle b_1,b_2\rangle$. We have $[P,B_0]=[P,b_1][P,b_2]$. Consider  now the subgroup  $\Omega_{1}([P,B_0])$. Obviously, $\Omega_{1}([P,B_0])$ can be viewed as a linear space of dimension at most two over the field with $p$ elements. It is well-known that the nilpotent subgroups of odd order of $GL(2,p)$ are abelian. Hence, we conclude that the derived group of $B_0$ centralizes $\Omega_{1}([P,B_{0}])$ and, by Lemma \ref{Lemmapre}(3), also centralizes $P$. Recall that $B_0$ is a subgroup generated by two arbitrarily chosen elements $b_1,b_2\in B$.  By Lemma \ref{foca} we  have  $H=\langle B\rangle$, and so  we conclude that $H'$ centralizes $P$. Let $\bar{G}=G/C_G(P)$. There is a natural action of $\bar{G}$ on $P$ and so we will view $\bar{G}$ as a group of automorphisms of $P$. We already know that $\bar{H}$ is abelian and it is clear that $\delta_j^*(\bar{G})=1$.

Suppose first that $\bar{H}$ is cyclic and choose an element $b\in B$ such that $\bar{H}$ is generated by $bC_G(P)$. We have  $P=[P,\bar{H}]=[P,b]$ which is cyclic, a contradiction.
Hence, $\bar{H}$ is not cyclic. Let $q$ be the prime such that $H$ has $q$-power order. By induction the group $\bar{G}$ contains a $q'$-element $y$ which is a power of $\delta_{j-2}^*$-commutator with the property that $[\bar{Q},y]$ is noncyclic. Moreover, Lemma \ref{copgen}(4) shows that $[\bar{Q},y]$ consists entirely of $\delta_{j-1}^*$-commutators. For any element $t\in[\bar{Q},y]$ we can choose $b_t\in B$ such that $[P,t]=[P,b_t]$. Therefore $[P,t]$ is cyclic for each $t\in[\bar{H},y]$. In view of Lemma \ref{bibi} this leads to a contradiction.

Now consider the case where $H$ is a $2$-subgroup. In this case the properties of towers listed before the lemma play an important role in our arguments. As before we have $[P,H]=P$ and we wish to show that $H$ contains a $\delta_{j-1}^*$-commutator $x$ with the property that $[P,x]$ is noncyclic. We can pass to the quotient $G/C_H(P)$ and assume that $H$ acts on $P$ faithfully. Choose a $\delta_{j-2}^{*}$-commutator $b\in P_{j-2}$. Suppose that $b$ normalizes an abelian subgroup $A$ in $H$. If $[A,b]\neq1$, then $[A,b]$ is a noncyclic abelian subgroup which, by Lemma \ref{copgen}(4), entirely consists of $\delta_{j-1}^{*}$-commutators. By Lemma \ref{bibi} $[P,x]$ is noncyclic for some $x\in [A,b]$ and we are done. Therefore $[A,b]=1$ for every abelian subgroup $A$ of $H$ which is normalized by $b$.

We know that $[h,\underbrace{b,\dots,b}_{j-1}]$ is a $\delta_{j-1}^{*}$-commutator for every $h\in H$. Therefore we can choose $a\in H$ such that $a$ and $[a,b]$ are both nontrivial $\delta_{j-1}^{*}$-commutators. If both $a$ and $[a,b]$ have order 2, then the subgroup $\langle a,[a,b]\rangle$ is abelian and consists of $\delta_{j-1}^{*}$-commutators. By Lemma \ref{bibi} $[P,x]$ is noncyclic for some $x\in\langle a,[a,b]\rangle$ and we are done. Therefore we can choose $a\in H$ such that $a$ and $[a,b]$ are both nontrivial $\delta_{j-1}^{*}$-commutators, the element $[a,b]$ being of order four. Since $a^{2}\in Z(H)$ and since $[Z(H),b]=1$, we have $[a^{2},b]=1$. So we have $$1=[a^{2},b]=[a,b][a,b]^{a}$$ 
and in particular  $a$ inverts $[a,b]$. It follows that $a$ normalizes $[P,[a,b]]$ which is a cyclic subgroup.  Now consider the action of the subgroup $D=\langle a,[a,b]\rangle$ on $[P,[a,b]]$. By Lemma \ref{autocyclic} $D'$ centralizes $[P,[a,b]]$. So in particular $[a,b]^{2}$ is nontrivial and it centralizes the cyclic subgroup $[P,[a,b]]$. Thus we get a contradiction by Lemma \ref{cycliccom}. The proof is now complete.   
\end{proof}

\begin{lemma}\label{mumba} Let $p$ be a prime, $j$ a positive integer and $G$ a group  such that $\delta_{j+1}^*(G)=1$. Let $P$ be the Sylow $p$-subgroup of $\delta_j^*(G)$ and assume that $[P,x]$ is cyclic for every $p'$-element $x$ which is a power of  a $\delta_{j-1}^*$-commutator. Then $P$ is cyclic.
\end{lemma}
\begin{proof} By passing to the quotient $G/O_{p'}(\delta_{j}^{*}(G))$ we may assume that $\delta_{j}^{*}(G)$ is a $p$-group and that $P=\delta_{j}^{*}(G)$.   If $P$ is abelian, the result is immediate from Lemma \ref{mumu}. Thus, we assume that $P$ is not abelian and use induction on the nilpotency class of $P$. We consider the quotient $G/Z(P)$ and by induction we conclude that $P/Z(P)$ is cyclic. However this implies that $P$ is abelian and we get  a contradiction.
\end{proof}

\begin{lemma}\label{murara} Let $p$ be a prime, $j$ a positive integer and $G$ a group such that $\delta_{j+1}^*(G)=1$. Suppose that $G$ possesses $m$ cyclic subgroups whose union contains all $\delta_j^*$-commutators of $G$ and  that the Sylow $p$-subgroup $P$ of $\delta_j^*(G)$ is nilpotent of class $c$. Let $x$ be a $p'$-element which is a power of $\delta_{j-1}^*$-commutator in $G$ such that $[P,x]$ is noncyclic. Then the order of $[P,x]$ is $(c,m)$-bounded.
\end{lemma}
\begin{proof}  The conjugation by the element $x$ induces a $p'$-automorphism  of $P$. Since every element of the form $[y,\underbrace{x,\dots,x}_{j}]$, with $y\in P$ is a $\delta_{j}^{*}$-commutator, Lemma \ref{coctionbounded} shows that the order of $[P,x]$ is $(c,m)$-bounded, as desired.
\end{proof}

\begin{lemma}\label{blibli} Let $j$ be a non-negative integer and $G$ a group such that $\delta_j^*(G)$ is nilpotent of class $c$. Suppose that $G$ possesses $m$ cyclic subgroups whose union contains all $\delta_j^*$-commutators of $G$.  Then $\delta_{j}^{*}(G)$ contains a subgroup $\Delta$ of $(c,m)$-bounded order which is normal in $G$ and has the property that $\delta_{j}^{*}(G)/\Delta$ is cyclic. If $j\geq 2$, then $\delta_{j}^{*}(G)=\Delta$.
\end{lemma}
\begin{proof} We argue by induction on $j$. Suppose first that $j=0$. In this case $G$ is nilpotent of class $c$ and it is covered by $m$ cyclic subgroups. The result is rather straightforward applying Lemma \ref{pri} to each Sylow subgroup of $G$.

So we assume that $j\geq1$. Let $P$ be a Sylow $p$-subgroup of $\delta_j^*(G)$ for some prime $p$. Denote by $\Delta_p$ the subgroup generated by all subgroups of the form $[P,y]$, where $y$ ranges over the set of all $p'$-elements which are powers of $\delta_{j-1}^*$-commutators in $G$ such that $[P,y]$ is noncyclic. By Lemma \ref{murara} the orders of all such subgroups $[P,y]$ have a common bound, which depends only on $c$ and $m$. We observe that $\Delta_p$ is a group which is nilpotent of class at most $c$ and is generated by elements of $(c,m)$-bounded order. Hence the exponent of $\Delta_p$ is $(c,m)$-bounded. Moreover, by Lemma \ref{copgen}(3) $\Delta_p$ is generated by $\delta_{j}^*$-commutators that are all contained in $m$ cyclic subgroups, and so we conclude that $\Delta_p$ has at most $m$ generators. It follows that the order of $\Delta_p$ is $(c,m)$-bounded. We further observe that since the bound on the order of $\Delta_p$ does not depend on $p$, it follows that $\Delta_p=1$ for all primes $p$ which are bigger than certain number depending only on $c$ and $m$. 

Let $\Delta$ be the product of the subgroups $\Delta_p$ over all prime divisors of $|\delta_j^*(G)|$. It is clear that $|\Delta|$ is  $(c,m)$-bounded. Consider the quotient $G/\Delta$. For simplicity, we just assume that $\Delta=1$. Then $[P,x]$ is cyclic for every $p'$-element $x$ which is a power of  a $\delta_{j-1}^*$-commutator. Then, by Lemma \ref{mumba}, $P$ is cyclic. Thus all Sylow subgroups of $\delta_{j}^{*}(G)$ are cyclic. It follows that $\delta_{j}^{*}(G)/\Delta$ is cyclic. Of course, if $j\geq 2$, then by Lemma \ref{coradical} we have  $\delta_{j}^{*}(G)=\Delta$.
\end{proof}

We are now ready to complete the  proof of Theorem \ref{deltathm}.

\begin{proof} Recall that $j\geq2$ and $G$ possesses $m$ cyclic subgroups whose union contains all $\delta_j^*$-commutators of $G$. We wish to show that the order of $\delta_j^*(G)$ is $m$-bounded. Let $C_1,\dots,C_m$ be the cyclic subgroups whose union contains all $\delta_j^*$-commutators of $G$. Without loss of generality we assume that each subgroup $C_i$ is generated by $\delta_j^*$-commutators (not necessarily by a single $\delta_j^*$-commutator). Thus, $\delta_j^*(G)=\langle C_1,\dots,C_m\rangle$ and in particular it follows that $\delta_j^*(G)$ can be generated by $m$ elements.  Let $x\in G$ be a $\delta_j^*$-commutator. For any $g\in G$ the conjugate $x^g$ is again a $\delta_j^*$-commutator and so $x^g\in C_i$ for some $i$. Since $C_i$ is cyclic, it contains only at most one subgroup of any given order and we conclude that the cyclic subgroup $\langle x\rangle$ has at most $m$ conjugates. Therefore the index of the normalizer of $\langle x\rangle$ in $G$ is at most $m$. Let $N$ be the intersection of all normalizers of cyclic subgroups generated by a $\delta_j^*$-commutator and  set $K=\delta_j^*(G)\cap N$. Since $\delta_j^*(G)$ is $m$-generated, it follows that the number of subgroups of $\delta_j^*(G)$ whose index is at most $m$ is $m$-bounded \cite[Theorem 7.2.9]{mhall} and so we deduce that the index of $K$ in $\delta_j^*(G)$ is $m$-bounded as well. It is clear that $K$ normalizes each of the subgroups $C_1,\dots,C_m$. This implies that $K$ is nilpotent of class at most 2. Indeed, since $Aut\, C_i$ is abelian for every $i=1,\ldots,m$, we deduce that $K/C_K(C_i)$ is abelian. So $K'$ centralizes $\delta_j^*(G)$ and therefore $K'\leq Z(K)$. 

Recall that given a group $G$, the last term of the upper central series of $G$ is called the hypercenter of $G$. It will be denoted by $Z_{\infty}(G)$.
Let us show that $K\leq Z_{\infty}(\delta_j^*(G))$.  Choose a Sylow $p$-subgroup $P$ of $K$. It is clear that $P$ is normal in $G$. If all the subgroups $C_i$ have $p$-power order, then all $\delta_{j}^{*}$-commutators of $G$ are $p$-elements and by \cite[Theorem 2.4]{forum} $G$ is soluble and $\delta_j^*(G)$ is a $p$-subgroup. Thus $\delta_j^*(G)$ is nilpotent and so,  we have $Z_{\infty}(\delta_{j}^{*}(G))=\delta_j^*(G)$ and the inclusion $P\leq Z_{\infty}(\delta_{j}^{*}(G))$ is clear.  Otherwise, choose a $p'$-element $x\in C_i$ for some $i$ which is a power of a $\delta_j^*$-commutator. Since $P$ normalizes $\langle x\rangle$, it follows that $x$ centralizes $P$. Therefore $\delta_j^*(G)/C_{\delta_j^*(G)}(P)$ is  a $p$-group and again the inclusion $P\leq Z_{\infty}(\delta_j^*(G))$ follows. Thus, $P\leq Z_{\infty}(\delta_j^*(G))$ for every prime $p$ and hence indeed $K\leq Z_{\infty}(\delta_j^*(G))$.

Therefore the index of $Z_{\infty}(\delta_j^*(G))$ in $\delta_j^*(G)$ is $m$-bounded. Thus, by Baer's Theorem \cite[Corollary 2, p.\ 113]{Robinson}, $\gamma_\infty(\delta_j^*(G))$ has $m$-bounded order. Passing to the quotient $G/\gamma_\infty(\delta_j^*(G))$ we can assume that $\delta_j^*(G)$ is nilpotent.  Hence $\delta_{j}^{*}(G)$ is the direct  product of its Sylow subgroups. It is sufficient to show that any Sylow subgroup of $\delta_{j}^{*}(G)$ has bounded order. Let us choose $p$ a prime that divides $|\delta_{j}^{*}(G)|$ and pass to the quotient $G/O_{p'}(\delta_j^*(G))$. So we assume that $\delta_j^*(G)$ is a $p$-group. In view of Lemma \ref{blibli} it is now sufficient to bound the nilpotency class of $\delta_j^*(G)$. It has already been mentioned that $K'$ centralizes $\delta_j^*(G)$ and therefore we can pass to the quotient $G/K'$ and, without loss of generality, assume that $K$ is abelian. Choose generators $x_1,\dots,x_m$ of the subgroups $C_1,\dots,C_m$ and let $t$ be the index of $K$ in $\delta_j^*(G)$. Since each subgroup $K\langle x_i\rangle$ is nilpotent of class at most 2 and since ${x_i}^t\in K$, it follows that $K^t$ centralizes $x_i$ for each $i=1,\dots,m$. In other words $K^t\leq Z(\delta_j^*(G))$. Passing again to the quotient $G/Z(\delta_j^*(G))$ we can assume that $K^t=1$. Since the index $t$ of $K$ in $\delta_j^*(G)$ is $m$-bounded and since $\delta_j^*(G)$ can be generated by $m$ elements, we conclude that the minimal number of generators for $K$ is $m$-bounded. Combining this with the fact that $K^t=1$, we immediately deduce that the order of $K$ and therefore that of $\delta_j^*(G)$ are $m$-bounded. Of course, this implies that so is the nilpotency class of $\delta_j^*(G)$. 
The proof is complete.
\end{proof}

%%%%%%%%%%%%%%%%%%%%%%%%%%%%%%%%%%%%%%%%%%%%%%%%%%%%%%%%%
\section{Theorem \ref{gammathm}}

In this  section we will deal with Theorem \ref{gammathm}. The proof  is similar to that of  Theorem \ref{deltathm} but in fact it is easier. Therefore we will not give a detailed proof here but rather describe only some steps. 

The next lemma is similar to Lemma  \ref{mumu}.
\begin{lemma}\label{mumugamma} Let $p$ be a prime and $G$ a  metanilpotent  group. Suppose that the Sylow $p$-subgroup $P$ of $\gamma_2^*(G)$ is abelian and noncyclic. Then there exists a $p'$-element $x$ with the property that $[P,x]$ is noncyclic.
\end{lemma}

\begin{proof}
 By Lemma \ref{metaHall} there is a Hall $p'$-subgroup  $H$ of $G$ such that $P=[P,H]$.  Now we consider the quotient $H/C_{H}(P)$ with acts faithfully on $P$. If $H/C_{H}(P)$ is noncyclic, then by Lemma \ref{bibi} there exists an element $x$ in $H$ such that $[P,x]$ is noncyclic. Therefore we assume that $H/C_{H}(P)$ is cyclic and let $x$ be an element in $H$ such that $xC_{H}(P)$ generates $H/C_{H}(P)$. Then $P=[P,x]$ is noncyclic and $x$ is the required element.
\end{proof}

The proof of the next lemma follows word-by-word that of Lemma \ref{mumba}. Therefore we omit the details.  
\begin{lemma}\label{mumbagamma} Let $p$ be a prime and $G$ a  metanilpotent  group. Let $P$ be the Sylow $p$-subgroup of $\gamma_2^*(G)$ and assume that $[P,x]$ is cyclic for every $p'$-element $x$. Then $P$ is cyclic.
\end{lemma}

The next results  are similar to Lemmas \ref{murara} and \ref{blibli}. Their proofs can be obtained   in the same way as those of Lemmas \ref{murara} and \ref{blibli}  with only obvious changes required. 

\begin{lemma}\label{muraragamma} Let $p$ be a prime, $j$ a positive integer and $G$ a  metanilpotent group. Suppose that $G$ possesses $m$ cyclic subgroups whose union contains all $\gamma_j^*$-commutators of $G$, and  that the Sylow $p$-subgroup $P$ of $\gamma_j^*(G)$ is nilpotent of class $c$. Let $x$ be a $p'$-element  in $G$ such that $[P,x]$ is noncyclic. Then the order of $[P,x]$ is $(c,m)$-bounded.
\end{lemma}

\begin{lemma}\label{blibligamma} Let  $j$ be a positive integer and $G$  a group  such that $\gamma_j^*(G)$ is nilpotent of class $c$. Suppose that $G$ possesses $m$ cyclic subgroups whose union contains all $\gamma_j^*$-commutators of $G$. Then $\gamma_j^*(G)$ contains a subgroup $\Delta$, of $(c,m)$-bounded order, which is normal in $G$ and has the property that $\gamma_{j}^{*}(G)/\Delta$ is cyclic.
\end{lemma}

From this we deduce  our Theorem \ref{gammathm}.
\begin{proof}[Proof of Theorem \ref{gammathm}]  Recall that $G$ has $m$ cyclic subgroups whose union contains all $\gamma_j^*$-commutators of $G$. We wish to prove that $\gamma_j^*(G)$ contains a subgroup $\Delta$, of $m$-bounded order, which is normal in $G$ and has the property that $\gamma_{j}^{*}(G)/\Delta$ is cyclic. Let $C_1,\dots,C_m$ be the cyclic subgroups whose union contains all $\gamma_j^*$-commutators of $G$. We assume that each subgroup $C_i$ is generated by $\gamma_j^*$-commutators. Thus, $\gamma_j^*(G)=\langle C_1,\dots,C_m\rangle$ and in particular it follows that $\gamma_j^*(G)$ can be generated by $m$ elements.   Let $N$ be the intersection of all normalizers of cyclic subgroups generated by a $\gamma_j^*$-commutator and $K=\gamma_j^*(G)\cap N$. Arguing as in the proof of Theorem \ref{deltathm} we deduce that the index of $K$ in $\gamma_j^*(G)$ is $m$-bounded. It is clear that $K$ is nilpotent of class at most 2 and $K\leq Z_{\infty}(\gamma_j^*(G))$. By Baer's Theorem $\gamma_\infty(\gamma_j^*(G))$ has $m$-bounded order. Passing to the quotient $G/\gamma_\infty(\gamma_j^*(G))$ we can assume that $\gamma_j^*(G)$ is nilpotent and, with further reductions, that $\gamma_j^*(G)$ is a $p$-group. In  view of Lemma \ref{blibligamma} it is now sufficient to bound the nilpotency class of $\gamma_j^*(G)$. Since $K'$ centralizes $\gamma_j^*(G)$, we can pass to the quotient $G/K'$ and without loss of generality assume that $K$ is abelian. Choose generators $x_1,\dots,x_m$ of the subgroups $C_1,\dots,C_m$ and let $t$ be the index of $K$ in $\gamma_j^*(G)$. Since each subgroup $K\langle x_i\rangle$ is nilpotent of class at most 2 and since ${x_i}^t\in K$, it follows that $K^t$ centralizes $x_i$ for each $i=1,\dots,m$. In other words $K^t\leq Z(\gamma_j^*(G))$. Passing again to the quotient $G/Z(\gamma_j^*(G))$ we can assume that $K^t=1$. Since the index $t$ of $K$ in $\gamma_j^*(G)$ is $m$-bounded and  $\gamma_j^*(G)$ can be generated by $m$ elements, we conclude that the minimal number of generators for $K$ is $m$-bounded. Combining this with the fact that $K^t=1$, we deduce that the order of $K$ and therefore that of $\gamma_j^*(G)$ are $m$-bounded. Of course, this implies that so is the nilpotency class of $\gamma_j^*(G)$. The proof is complete.
\end{proof}

%%%%%%%%%%%%%%%%%%%%%%%%%%%%%%%%%%%%%%%%%%%%%%%%%%%%%%%%%%


\begin{thebibliography}{00}
\bibitem{AST} C.\ Acciarri, P.\ Shumyatsky, A.\ Thillaisundaram, Conciseness of coprime commutators in finite groups, \textit{Bulletin of the Australian Mathematical Society}, doi:10.1017/S0004972713000361, to appear. 

\bibitem{focal} C. Acciarri, G.\,A. Fern\'andez-Alcober, P.\,Shumyatsky, A focal subgroup theorem for outer commutator words, \textit{J. Group Theory} \textbf{15} (2012), 397--405.

\bibitem{CN}G. Cutolo, C. Nicotera, Verbal sets and cyclic coverings. \textit{J. Algebra} {\bf 324} (2010) 1616--1624.

\bibitem{FS}G.\,A. Fern\'andez-Alcober, P. Shumyatsky, On groups in which commutators are covered by finitely many cyclic subgroups. \textit{J. Algebra} {\bf319} (2008) 4844--4851.

\bibitem{Gorenstein} D. Gorenstein, \ {\it Finite Groups},  Chelsea Publishing Company, New York, 1980.
\bibitem{mhall} M. Hall, Jr., {\it The Theory of Groups}, The Macmillan Co., New York, 1959.

\bibitem{lost} M.\,W. Liebeck, E.\,A. O' Brien, A. Shalev, P.\,H. Tiep, The Ore conjecture, \textit{J. Eur. Math. Soc. (JEMS)} \textbf{12} (2010), no. 4, 939--1008.

\bibitem{neumann} B.\,H. Neumann, Groups covered by permutable subsets, \textit{J. London Math. Soc.} {\bf 29} (1954), 236--�248.
\bibitem{pellegrini} M.\,A. Pellegrini, P. Shumyatsky, Coprime commutators in $\mathrm{PSL}(2,q)$ , \textit{Arch. Math.} \textbf{99} (2012), 501--507. 

\bibitem{Robinson} D.\,J.\,S. Robinson, \textit{Finiteness Conditions and Generalized Soluble Groups}, Part 1, Springer-Verlag, 1972.

\bibitem{forum} P.\,\,Shumyatsky, Commutators of elements of coprime orders in finite groups, \textit{Forum Mathematicum}, doi:10.1515/forum-2012-0127, to appear.


\bibitem{tur} A. Turull, Fitting height of groups and of fixed points, \textit{J. Algebra} \textbf{86} (1984), 555-566. 



 \end{thebibliography}
\end{document}